\documentclass[preprint]{imsart}
\RequirePackage[OT1]{fontenc}
\RequirePackage{amsthm,amsmath}
\RequirePackage[numbers]{natbib}
\RequirePackage{lscape}

\pubyear{\today}

\usepackage{amssymb, amsmath, natbib, amsthm, graphicx, euscript, mathrsfs }

\newcommand{\eqd}{\stackrel{\mathcal{L}}{=}}

\newcommand{\diag}{\operatorname{diag}}
\newcommand{\T}{\mathcal{T}}
\newcommand{\E}{{\mathbb E}}

\newcommand{\tW}{\widetilde{W}}

\def\R{I\!\!R}

\newcommand{\Var}{\operatorname{Var}}
\newtheorem{thm}{Theorem}[section]

\newtheorem{rem}[thm]{Remark}

\theoremstyle{definition}
\newtheorem{defi}[thm]{Definition}

\renewcommand{\P}{{\mathbb P}}
\renewcommand{\L}{\mathscr{L}}

\newcommand{\B}{\mathcal{B}}

\newcommand{\vX}{\vec{X}}
\newcommand{\U}{\widehat{U}}

\newcommand{\s}{s}

\begin{document}
\begin{frontmatter}
\title{Series representations for bivariate \\time-changed  L{\'e}vy models\thanksref{T1}}
\runtitle{Series representations for time-changed L{\'e}vy models}

\begin{aug}
\author{\fnms{Vladimir} \snm{Panov}\ead[label=e1]{vpanov@hse.ru}}
\and
\author{\fnms{Igor} \snm{Sirotkin}\ead[label=e2]{sirotkinig@gmail.com}}

\address{Laboratory of Stochastic Analysis and its Applications \\National Research University Higher School of Economics\\ 
Shabolovka  31, building G,  115162 Moscow, Russia\\
\printead{e1,e2}}

\thankstext{T1}{This study (research grant No 14-05-0007) was supported by the National Research University-Higher School of Economics' Academic Fund Program in 2014.}

\runauthor{V.Panov  and I.Sirotkin}
\affiliation{Higher School of Economics}
\end{aug}

\begin{abstract}
In this paper, we analyze a L{\'e}vy model based on two popular concepts  - subordination and L{\'e}vy copulas. More precisely, we consider a two-dimensional L{\'e}vy process such that each component is a time-changed (subordinated) Brownian motion and the dependence between subordinators  is described via some L{\'e}vy copula.  We prove a series representation for our model, which can be efficiently used for simulation purposes, and provide some practical examples based on real data. 
\end{abstract}



\tableofcontents
\end{frontmatter}

\section{Introduction} 
Copula is with no doubt the most popular tool for describing the dependence  between two random variables.    The popularity is partially based on the fact that the dependence between any random variables can be modelled by some copula. This fact is known as  Sklar's theorem, which states that for any two random variables \(Y_{1}\) and \(Y_{2}\) there exists a \textit{copula} \(C\) (a two-dimensional real-valued distribution function with domain \([0,1]^{2}\) and uniform margins) such that 
   \begin{eqnarray}
\label{Copula}
   \P \left\{Y_{1} \leq u_{1}, Y_{2} \leq u_{2}\right\} 
   =
   C \Bigl( 
   \P\left\{
   	 Y_{1} \leq u_{1}
   \right\},
   \P\left\{
   	 Y_{2}\leq u_{2}
   \right\}
   \Bigr),    
\end{eqnarray} 
for any \(u_{1}, u_{2} \geq 0.\) We refer to  Cherubini et al. (\citeyear{cher}), Joe (\citeyear{Joe}), Nelsen(\citeyear{Nelsen}) for a comprehensive overview of the copula theory. 
    
 Now let us switch from random variables to stochastic processes and try to describe dependence between components of some two-dimensional \textit{L{\'e}vy process} \(\vX(t)= \left(X_{1}(t),X_{2}(t)\right)\), that is, of some  cadlag process with independent and stationary increments. Applying Sklar's theorem for any fixed time moment \(t\), we get that the dependence between \(X_{1}\) and \(X_{2}\) can be described by some copula \(C_{t}\), i.e., 
\begin{eqnarray}
\label{Ct}
   \P \left\{X_{1}(t) \leq u_{1}, X_{2}(t) \leq u_{2}\right\} 
   =
   C_{t} \Bigl( 
   \P\left\{
   	 X_{1}(t) \leq u_{1}
   \right\},
   \P\left\{
   	 X_{2}(t)\leq u_{2}
   \right\}
   \Bigr),    
\end{eqnarray} 
for any \(u_{1}, u_{2} \geq 0.\) Nevertheless, the direct application of the representation \eqref{Ct} to stochastic modeling has a couple of drawbacks.  First, it turns out that the copula \(C_{t}\) in most cases essentially depends on \(t\), see Tankov (\citeyear{Tankov})  for examples. Second, 
since the distribution of \(\vX(t)\) is infinitely divisible,  \eqref{Ct} is possible only for some subclasses of copulas.  In other words,  the class \(C_{t}\) depends on the class of marginal laws \(X_{i}(t)\). 
 
To avoid such difficulties,  researches are trying to characterize the dependence between the components of L{\'e}vy process in \textit{the time-independent fashion}.  One of the most popular approaches  for this characterization is the so-called L{\'e}vy copula (defined below), which was introduced by Tankov (\citeyear{Tankov03}), and later studied by Barndorff-Nielsen and Lindner (\citeyear{BNLindner}),  Cont and Tankov (\citeyear{ContTankov}), Kallsen and Tankov (\citeyear{KallsenTankov}), and others.  Among many papers in this field, we would like to emphasize some articles about statistical inference for L{\'e}vy copulas (mainly with applications to insurance), which include some basic ideas that are widely used in statistical research on this topic, in particular, in the  statistical analysis in  the current paper -  Esmaeili and Kl{\"u}ppelberg (\citeyear{Esm} and \citeyear{Esm2}), Avanzi, Cassar and Wong (\citeyear{Avanzi}), B{\"u}cher and Vetter (\citeyear{BV}).

The main objective of this article is the application of the L{\'e}vy copula approach to a class of stochastic processes, known as \textit{time-changed L{\'e}vy processes}. In the one-dimensional case, the time-changed L{\'e}vy process is defined as $Y_s = L_{\T(s)},$ where  \(L\) is a L{\'e}vy process and \(\T\)  is a non-negative, non-decreasing stochastic process with \( \T(0)=0\) referred as \textit{stochastic time change} or simply \textit{stochastic clock}.   If the process \(\T\) is also a L{\'e}vy process, then it is called the \textit{subordinator}, and the process \(Y_{s}\) is usually refered as the subordinated process.  The economical interpretation of the time change is based on the idea that  the ``business'' time \(\T(s)\) may run faster than the physical time in some periods, for instance, when the amount of transactions is high, see Clark (\citeyear{Clark}), An{\'e} and Geman (\citeyear{AneGeman}), Veraart and Winkel (\citeyear{VW}).  

In this paper, we consider one natural generalization of the aforementioned model to the two-dimensional case. Our construction  is  based on the so-called \textit{multivariate subordination},  introduced by  Barndorff-Nielsen,  Pederson and Sato (\citeyear{BNPedSato}). In particular, we prove the series representation of the processes from our class, which allows to simulate the processes with given characteristics, and show an application of this representation to real data.

The paper is organized as follows. In the next two sections, we shortly explain the notion of L{\'e}vy copula and the idea of stochastic change of time. Afterwards, in Section~\ref{2dtc}, we introduce our model and discuss some properties of it. Our main results are given in Section~\ref{mainres}, where we also provide some examples. In the last section, we apply our model to some stock prices.

\section{L\'evy copulas}
The construction of the L{\'e}vy copula is based on the concept of tail integrals. 
\begin{defi}
\label{tail}
For a one-dimensional L{\'e}vy process \(Z\) with L{\'e}vy measure \(\nu_{Z}\), tail integral is defined as 
\begin{eqnarray*}
U_{Z} (x) :=
\begin{cases}
		 \nu_{Z}\left(x, +\infty\right), &\text{if $x>0$,}\\
		-\nu_{Z}\left(-\infty, x \right), &\text{if $x<0$.}
	\end{cases}
\end{eqnarray*}
\end{defi}
\!\!Definition~\ref{tail} can be equivalently written as
\begin{eqnarray*}
U_{Z} (x) :=	
 (-1)^{\s(x)} \nu_{Z}\Bigl( I(x)\Bigr),
\end{eqnarray*}
where
\begin{eqnarray*}
	I(x) :=\begin{cases}
		 \left(x, +\infty\right), &\text{if $x>0$,}\\
		\left(-\infty, x \right), &\text{if $x<0$,}
	\end{cases}
	\qquad \mbox{and} \qquad 
	s(x):=\begin{cases}
		 2, &\text{if $x>0$,}\\
		1, &\text{if $x<0$.}
	\end{cases}
\end{eqnarray*}
The reason for this definition is that  in the case of infinite measure \(\nu_{Z}\), \(U_{Z} (A) \) is infinite for any set \(A\) which contains \(0\).  Analogously, for a two-dimensional process \(\vec{Z} = \left( Z_{1}, Z_{2} \right)\) with L{\'e}vy measure \(\nu_{\vec{Z}}\), 
\begin{eqnarray*}
U_{\vec{Z}} (x_{1}, x_{2}) := (-1)^{s(x_{1}) + s(x_{2})}\nu_{\vec{Z}} \Bigl( I(x_{1}) \times  I(x_{2})\Bigr),
\end{eqnarray*}
and this definition is also correct for any real \(x_{1}\) and \(x_{2}\).
\begin{defi} \label{defii}
A two-dimensional L{\'e}vy copula is  a function from \(\bar{\R}^{2}\) to \(\bar{\R}\)  such that 
\begin{enumerate}
\item \(F\) is grounded, that is, \(F\left(u_{1},u_{2}\right) =0\) if \(u_{i}=0\) for at least one \(i=1,2\).
\item \(F\) is 2-increasing.
\item \(F\) has uniform margins, that is, \(F(u, \infty) = F(\infty,u) =u\).
\item \(F\left(u_{1}, u_{2}\right) \ne \infty \) for  \(\left( u_{1}, u_{2} \right) \ne \left(\infty, \infty\right)\).
\end{enumerate}
\end{defi}

\!\!The main result on L{\'e}vy copulas is an analogue of the Sklar theorem for ordinary copulas which states that for any two-dimensional L{\'e}vy process \(\vec{Z}\) with tail integral \(U_{\vec Z}\) and marginal tail integrals \(U_{Z_{1}}\) and \(U_{Z_{2}}\), there exists a L{\'e}vy copula \(F\) such that  
\begin{eqnarray}
\label{Sklar}
	U_{\vec{Z}} (x_{1}, x_{2}) = F \left( U_{Z_{1}} (x_{1}) , U_{Z_{2}} (x_{2}) \right),
\end{eqnarray}
and vise versa, for any L{\'e}vy copula \(F\) and any one-dimensional L{\'e}vy process with tail integrals \(U_{Z_{1}}\) and \(U_{Z_{2}}\) there exists a two-dimensional L{\'e}vy process with tail integral \(U_{\vec{Z}}\) given by \eqref{Sklar} with marginal tail integrals \(U_{Z_{1}}\) and \(U_{Z_{2}}\) . The first part of this theorem can be easily verified for the case when the one-dimensional L{\'e}vy measures are infinite and have no atoms, because in this case the L{\'e}vy copula is equal to 
\begin{eqnarray}
\label{Levycopula}
	F(u_{1},u_{2}) = U \Bigl( 
		U_{Z_{1}}^{-1} (u_{1}), U_{Z_{2}}^{-1} (u_{2})
	\Bigr),
\end{eqnarray}
where \(U\) is the tail integral of the L{\'e}vy measure of  \(\vX(t)\), see \cite{KallsenTankov}. 

%
%
%

\section{Time-changed L{\'e}vy models}
As it was already mentioned in the introduction, the time-changed L{\'e}vy process in the one-dimensional case is defined as 
\begin{eqnarray}
\label{tc}
Y_{s}=L_{\T(s)},
\end{eqnarray}
where \(L\) is a L{\'e}vy process, and \(\T(s)\) - a non-negative, non-decreasing stochastic process with \(\T(0)=0\). This class of models has strong mathematical background based on the so-called Monroe theorem \cite{Monroe}, which stands that any semimartingale can be represented as a time-changed Brownian motion (that is, in the form \eqref{tc} with \(L\) equal to the Brownian motion \(W\)) and vise versa, any time-changed Brownian motion is a semimartingale. Various aspects of this theory are discussed in \cite{BNS} and \cite{Cherny}. Nevertheless, the first part of the Monroe theorem doesn't hold if one introduces any of the following additional assumptions: 
\begin{enumerate}
\item  Processes \(W\) and \(\T\) are independent. This assumption is widely used in the statistical literature and 
is quite convenient for both theoretical and practical purposes.
\item Time change process \(\T\) is itself a L{\'e}vy process; such processes are known as \textit{subordinators}. In this case any resulting process \(Y_{s}\) is also a L{\'e}vy process, which is usually called \textit{a subordinated process}. 
\end{enumerate}
These drawbacks of the time-changed Brownian motion lead to the idea of considering more general model \eqref{tc} with any L{\'e}vy process instead of the Brownian motion and introducing the assumption that the processes \(\T\) and \(L\)  are independent. This  model has been attracting attention of many researches, see, e.g., \cite{panov2013d}, \cite{Bertoin} \cite{Carr}, \cite{Cherubini},  \cite{Schoutens}. 

Nevertheless, there is no clear understanding in the literature how to extend this model to the two-dimensional case. The most 
popular construction is to consider the model \eqref{tc} with a two-dimensional L{\'e}vy process \(\vec{L}\) and to provide a time change  in each component with the same  process \(\T\), see \cite{Sato}.  

Interestingly enough, in the case when \(\vec{L}\)  is a Brownian motion,  the correlation coefficient between subordinated processes is upper bounded by the correlation coefficient between the components of the Brownian motion, see  \cite{Eberlein}. Moreover, these coefficients coincides in some cases, see \cite{ContTankov}.

\section{Two-dimensional subordinated processes}
\label{2dtc}
In this section, we introduce a two-dimensional generalization of the model \eqref{tc}. This generalization is based on the notion of the two-dimenstional subordinator, which we define below.
\begin{defi}
 \textit{A two-dimensional subordinator} \(\vec{T}(s)=\left(T_{1}(s), T_{2}(s) \right)\)  is a L{\'e}vy process in \(\R^{2}\) such that both components \(T_{1}\) and \(T_{2}\) are one-dimensional subordinators. 
\end{defi}
Evidently, \(\vec{T}\) is a two-dimensional subordinator if its margins \(T_{1}\) and \(T_{2}\)  are independent. Another example is given by the following statement (see \cite{Semeraro}): if \(T_{1}, T_{2}, T_{3}\) are 3 independent subordinators, then the processes  \[\left( T_{1}(s) +  T_{3}(s), \;\; T_{2}(s) + T_{3}(s)\right)\] and \[\left(T_{1}(T_{3}(s)), \;\;T_{2}(T_{3}(s))\right)\]are two-dimensional subordinators. To the best of our knowledge,  general criteria providing necessary and sufficient conditions for two-dimensional process with marginal positive L{\'e}vy processes  to be a two-dimensional subordinator,  are not known in the literature.

Consider now a two-dimensional L{\'e}vy process \(\vec{L}(t)=\left(L_{1}(t), L_{2}(t) \right)\) with independent components and  a two-dimensional subordinator \(\vec{\T}(s)=\left(\T_{1}(s), \T_{2}(s) \right)\) such  that \(\T_{i}(s)\) is independent of \(L_{i}(s)\) for any \(i=1,2\).  Define the subordinated process by composition 
\begin{eqnarray}
\label{main}
  \vX (s) = \Bigl( X_{1}(s), X_{2} (s) \Bigr) := \Bigl( L_{1}(\T_{1}(s)), L_{2}(\T_{2}(s)) \Bigr).
\end{eqnarray}
 This construction, known as multivariate subordination,  was firstly considered in \cite{BNPedSato}.
  The next theorem sheds some light to the characteristics of such processes. 
 \begin{thm}
 \label{thm1}
 Let \(W_{i}, \: i=1,2\) be two independent one-dimensional Brownian motions and let  \(\vec{\T}(s) = \left(\T_{1}(s), \T_{2}(s) \right)\) be a two-dimensional  subordinator with L{\'e}vy triplet \(\left(\vec{\rho}, 0, \eta\right)\), where \(\vec{\rho} = \left(\rho_{1}, \rho_{2}\right)\) with \(\rho_{i} \geq 0,\;  i=1,2\) and \(\eta\) is a L{\'e}vy measure in \(\R_{+}^{2}\). 
 
 Denote by \(\diag\left(x, y\right)\) with \(x,y \in \R\) a two-dimensional diagonal matrix with the values \(x\) and \(y\) on the diagonal. 
 
 Then
 \begin{itemize}
\item  the process 
\begin{eqnarray}
\label{vX} 
\vX(s):=\Bigl( W_{1}(\T_{1}(s)), W_{2}(\T_{2}(s)) \Bigr)
\end{eqnarray}  is a two-dimensional L{\'e}vy process;
\item    the L{\'e}vy triplet of the process \(\vX\) is given by
\[
\Bigl(
\vec{0}, \diag\left(\rho_{1}, \rho_{2}\right),
\nu
\Bigr),
\]
where the L{\'e}vy measure \(\nu\) is defined as 
\[
	\nu(B):= \int_{\R_{+}^{2}}\mu \Bigl(B; \; \diag(y_{1}, y_{2}) \Bigr) \; \eta(dy_{1}, dy_{2}), 
	\quad B \subset \R^{2},
\]	 
and \(\mu\left( B ; A  \right):=\P \left\{ \xi_{A} \in B\right\}\)  is the probability that a random variable \(\xi_{A}\) with zero mean and covariance matrix \(A\) belongs to the set \(B\).
\end{itemize}
 \end{thm}
 \begin{proof}
 This theorem is essentially proven in \cite{BNPedSato}.
 \end{proof}
 \section{Series representation for subordinated processes} 
 \label{mainres}
 In this section, we apply the result by Rosinsky \cite{Rosinsky} to our setup. 
 \begin{thm}
 \label{thm2}
 	Let \(\vec{X}(s)\) be a two-dimensional L{\'e}vy process constructed by multivariate subordination of the Brownian motion, see Theorem~\ref{thm1} for notation. Denote by  \(F(u,v)\) a positive L{\'e}vy copula between \(\T_{1}(s)\) and \(\T_{2}(s)\). Assume that \(F(u,v)\) is continuous and the mixed derivative \(\partial^{2} F (u,v)/ \partial u \partial v\) exists in \(\R_{+}^{2}\). Moreover, assume that there exists a density function \(p^{*}(\cdot)\) and a function \(f^{*}(u,x): \R_{+}^{2} \to \R_{+}\), such that 
\begin{enumerate}
\item for any \(u,x>0\),
\begin{eqnarray}
\label{fp}
	\int_{-\infty}^{f^{*}(u, x)} p^{*}(z) dz = \frac{\partial F(u,x) }{\partial u};
\end{eqnarray}
\item the function \(f^{*}(u, x)\) monotonically increases in \(x\) for any fixed \(u\), and moreover the equation
\begin{eqnarray*}
	f^{*}(u,x) = y
\end{eqnarray*}
 has a solution in closed form with respect to \(x\) for any \(y>0\);  we denote this solution by \(h^{*}(u, y)\).
 \end{enumerate}
Next, define a two-dimensional stochastic process \(\vec{Z}(s)= \left( Z_{1} (s), Z_{2}(s)\right)\):
	\begin{eqnarray}
		\label{res1}
		Z_{1} (s) &:=&  \sum_{i=1}^{\infty} \sqrt{U_{1}^{(-1)}(\Gamma_{i})} \cdot  G^{(1)}_{i} \cdot I\left\{ R_{i} \leq s \right\},\\ 
		\label{res2}
		Z_{2} (s)  &:=&  \sum_{i=1}^{\infty}   G^{(2)}_{i}  \sqrt{ U_{2}^{(-1)}\left(
			h^{*}(\Gamma_{i},G^{(3)}_{i} )
		\right)}  \cdot  I\left\{ R_{i} \leq s \right\}
	\end{eqnarray}
 where \(U_{1}\) and \(U_{2}\)  are tail integrals of the subordinators \(T_{1}\) and \(T_{2}\) resp., 
 \(U_{1}^{(-1)}\) and \(U_{2}^{(-1)}\) are their generalized inverse functions, that is,
 \begin{eqnarray*}
	U_{i}^{(-1)} (y) =  \inf \left\{ 
		x>0: \quad U_{i}(x) <y 
	\right\}, \quad i=1,2, \quad y \in \R_{+},
\end{eqnarray*}
 \(\Gamma_{i}\) is an independent sequence of jump times of a standard Poisson process, \(G^{(1)}_{i}, \; G^{(2)}_{i} \) - are two sequences of i.i.d. standard normal r.v., \(G^{(3)}_{i}\) - sequence of i.i.d. random variables with density \(p^{*}(\cdot)\),  \(R_{i}\) -sequence of i.i.d. r.v., uniformly distributed on \([0,1]\), and all sequences of r.v. are independent of each other.   
 Then 
	\[\vec{X}(s) \eqd \vec{Z}(s), \qquad \forall s \in [0,1].\]
  \end{thm}
\!\!\underline{\textbf{Examples.}} 
\begin{enumerate}
\item
\label{clevy}
Consider the positive Clayton-L{\'e}vy copula 
\begin{eqnarray*}
	F_{C}(u,v) = (u^{-\theta}+ v^{-\theta})^{-1/\theta}
\end{eqnarray*}
with some \(\theta>0\).  First derivative of this function is equal to 
\begin{eqnarray*}
\frac{\partial F_{C}(u,v)}{\partial u} = \frac{(v/u)^{\theta+1}}
{
	\left(
		1+(v/u)^{\theta}
	\right)^{(1+\theta)/\theta}
}.
\end{eqnarray*}
Motivated by this representation, we suggest to define the density function \(p^{*}(z)\) as
\begin{eqnarray*}
	p^{*}(z) = \frac{\partial}{\partial z} \left\{
	\frac{z^{\theta+1}}
{
	\left(
		1+z^{\theta}
	\right)^{(1+\theta)/\theta}
}
\right\}
=
\frac{\partial}{\partial z} \left\{
	\frac{1}
{
	\left(
		z^{-\theta}+1
	\right)^{(1+\theta)/\theta}
}
\right\},
\end{eqnarray*}
and the function \(f^{*}(u,x):= x/u\). Both conditions on the functions \(p^{*}\) and \(f^{*}\) are fulfilled. 
\item  Note that the same arguments can be applied to  any sufficiently smooth homogeneous L{\'e}vy copula, that is, to any copula such that
\begin{eqnarray}
	\label{LC}
	F_{H}(k u, k v) = k F_{H}(u,v), \qquad \forall \;  u,v,k>0,
\end{eqnarray}
see Remark~\ref{cop} about the difference between ordinary copulas and L{\'e}vy  copulas. Note that taking the derivatives with respect to \(u\) from both parts of \eqref{LC}, yields 
\begin{eqnarray*}
	\frac{\partial} {\partial u} F_{H}(k u, k v) = 
	\frac{\partial} {\partial u} F_{H}(u, v) = 
	\left.
	\frac{\partial} {\partial r_{1}} F_{H}(r_{1},r_{2}) 
	\right|_{\substack {r_{1}= 1 \\ r_{2}=v/u}},
\end{eqnarray*}
and therefore one can define 
\begin{eqnarray*}
		p^{*}(z) =  \frac{\partial}{\partial z} \left\{
				\left.
					\frac{\partial} {\partial r_{1}} F_{H}(r_{1},r_{2}) 
					\right|_{\substack {r_{1}= 1 \\ r_{2}=z}}
		\right\}, \qquad f^{*}(u,x):= x/u.
\end{eqnarray*}
For the description of the class of  homogeneous L{\'e}vy copulas we refer to Section~4 from \cite{BNLindner}.
\item
Moreover, we can apply the same approach for any mixtures of homogeneous L{\'e}vy copulas. In fact, consider the function 
\begin{eqnarray*}
 	F_{M}(u,v) = \sum_{r=1}^{n} \beta_{r} F_{r}(u,v), 
\end{eqnarray*}
where \(\beta_{1}, ..., \beta_{n}\) are positive numbers such that \(\sum_{r=1}^{n} \beta_{r}=1\) and \(F_{r}(u,v)\) are homogeneous L{\'e}vy copulas for any \(r=1..n\). It's easy to see that \(F_{M}(u,v)\) is a homogeneous L{\'e}vy copula. 
 As it was shown in the previous example, we can take \(f^{*}(u,x):= x/u\). Note also that in the case of mixture model,
\begin{eqnarray*}
	p^{*}(z) = \sum_{r=1}^{n} \beta_{r} p_{r}^{*}(z),
\end{eqnarray*}
where by \(p_{r}^{*}(\cdot)\) we denote the density functions such that
\begin{eqnarray*}
	\int_{-\infty}^{x/u} p_{r}^{*}(z) dz = \frac{\partial F_{r}(u,x) }{\partial u}.
\end{eqnarray*}

\end{enumerate}
 \begin{proof} (of Theorem~\ref{thm2})
Since the L{\'e}vy copula \(F\)  is sufficiently smooth, we can differentiate both parts in \eqref{Sklar} and get that 
 \begin{eqnarray*} 
 \nu (B) &=& \int_{\R_{+}^{2}}   \mu\Bigl(B \; ; \diag(y_{1}, y_{2})  \Bigr) \; 
 \left.\frac{\partial^{2} F}{\partial r_{1} \; \partial r_{2}}\right|_{\substack {r_{1}= U_{1}(y_{1}) \\ r_{2}=U_{2}(y_{2})}} \;
 d\left( U_{1}(y_{1})\right)
 d\left( U_{2}(y_{2})\right) \\
 &=& 
 \int_{\R_{+}^{2}}   \mu\Bigl( B \; ;  \diag(U_{1}^{-1} (r_{1}), U_{2}^{-1} (r_{2})) \Bigr) \; 
\frac{\partial^{2} F (r_{1}, r_{2})}{\partial r_{1} \; \partial r_{2}} \; dr_{1} dr_{2},
 \end{eqnarray*}
see Proposition~5.8 from \cite{ContTankov}. Our general aim is to apply the result by Rosinsky \cite{Rosinsky}, which is  nicely formulated as Theorem~6.2 in \cite{ContTankov}. Comparison of this result and the statement of our theorem leads to the idea to find a function \(H: \left(0, +\infty\right) \times S \to \R^{2}\), where \(S\) is a measurable space, such that 
 \begin{eqnarray}
 \label{nuB}
 \nu(B) = \int_{\R_{+}} 
 	\P \Biggl\{  \vec{H}\left(r, \vec{D}\right) \in B \Biggr\}
	dr, \qquad \forall \;B \in \B (\R^{2}),
 \end{eqnarray}
 where \(\vec{D}\) is a random vector from \(S\). 
 
 First note that it is sufficient to consider the sets \(B = B_{1} \times B_{2}\), where \(B_{1} = [x, \infty), \; B_{2} = [y, \infty), \; x,y \in \R\). For such \(B\), 
 \begin{eqnarray*}
\mu\Bigl( B; \;\diag(U_{1}^{-1} (r_{1}), U_{2}^{-1} (r_{2}))\Bigr) = \mu\Bigl( B_{1} \; ; U_{1}^{-1} (r_{1})\Bigr)  \cdot \mu\Bigl( B_{2} \; ; U_{2}^{-1} (r_{2})\Bigr),
 \end{eqnarray*}
 where by \(\mu(\sigma, B)\) we denote the one-dimensional normal distribution with zero mean and variance equal to \(\sigma\) (since there is no risk of confusion, we use the same letter for 2-dimensional and 1-dimensional distributions). Therefore, 
  \begin{eqnarray}
  \label{nu1}
	 \nu(B) 
	  &=&
	   	\int_{\R_{+}} 
		\E_{\L(r_{1})}
		 \Biggl[
			\mu\Bigl( B_{2} \; ; U_{2}^{-1} (\cdot) \Bigr) 
		 \Biggr]
\mu\Bigl(  B_{1} \; ; U_{1}^{-1} (r_{1}) \Bigr) 
	 dr_{1},
 \end{eqnarray} 
 where by 
\begin{eqnarray*}
 	\E_{\L(r_{1})}
		 \Biggl[
			\mu\Bigl(  B_{2}\; ; U_{2}^{-1} (\cdot) \Bigr) 
		 \Biggr]
	&=& 
		\int_{\R_{+}}
		\mu\Bigl(B_{2}  \; ;  U_{2}^{-1} (r_{2})\Bigr) 
		\;
		\frac{\partial}{\partial r_{2}} \left( \frac{\partial F (r_{1}, r_{2})}{\partial r_{1}} \right) \;  dr_{2} 
\end{eqnarray*}
we denote the mathematical expectation with respect to the measure \(\L(r_{1})\) with the distribution function \(\tilde{F}(r_{2}) = \partial F (r_{1}, r_{2}) / \partial r_{1}\) (the statement that \(\tilde{F}(r_{2}) \) is in fact a distribution function is proven in \cite{ContTankov}, Lemma~5.3). By the well-known Fubini theorem, 
\begin{eqnarray*}
 	\E_{\L(r_{1})}
		 \Biggl[
			\mu\Bigl( B_{2} \; ; U_{2}^{-1} (\cdot) \Bigr) 
		 \Biggr]
	&=& 
		\int_{B_{2}} g(v \; ; r_{1}) \; dv,
\end{eqnarray*}
where
\begin{eqnarray}
 \label{g}
 g(\cdot; \; r_{1}) = \; 	\int_{\R_{+}}
			p_{1}\Bigl( \cdot \; ;
				U_{2}^{-1} (r_{2})
			\Bigr)
	 	 	\frac{\partial^{2} F (r_{1}, r_{2})}{\partial r_{1} \partial r_{2} }
			dr_{2}, \qquad r_{1}>0,
\end{eqnarray}
and \(p_{1}\left( \cdot \; ;  U_{2}^{-1} (r_{2})\right) \) is the  density of the normal distribution with zero-mean and variance equal to \(U_{2}^{-1} (r_{2})\), and  \(\partial^{2} F (r_{1}, r_{2}) / (\partial r_{1} \partial r_{2})\) is the  density function corresponding to the distribution function \(\tilde{F}(r_{2})\).

Note that  \(g\) is a density function, see Remark~\ref{remm}. Changing the variables we get 
\begin{eqnarray*}
	g(\cdot; \; r_{1}) =   \int_{\R_{+}}
			p_{1}\Bigl( \cdot \; ;
				 \tilde{r}_{2}
			\Bigr)
			\left. 
		 	 	\frac{\partial^{2} F (r_{1}, r_{2})}{\partial r_{1} \partial r_{2} }			
			\right|_{r_2 = U_{2}\left( \tilde{r}_{2}\right)}
			d \left( 
				 U_{2}\left( \tilde{r}_{2}\right)
			\right), \qquad r_{1}>0.
\end{eqnarray*}
The last expression yields that \(g(\cdot; \; r_{1}) \) is in fact  a variance mixture of the normal distribution (see \cite{BNKS} or \cite{Kelker}). This in particularly gives that the random variable 
\begin{eqnarray*}
	\xi = \eta_{1} \sqrt{\eta_{2}}
\end{eqnarray*}
has a distribution with density \(g(\cdot; r_{1})\), where \(\eta_{1}\) has standard  normal distribution, and \(\eta_{2}\) - distribution with density 
\begin{eqnarray*}
	\breve{p}(\cdot; r_{1})  &=&  - 
	\left. 
		\frac{\partial^{2} F (r_{1}, r_{2})}{\partial r_{1} \partial r_{2} }
	\right|_{r_{2}=U_{2}(\cdot)}
			U'_{2}(\cdot)\\
			&=&
			\left. 
		\frac{\partial^{2} F (r_{1}, r_{2})}{\partial r_{1} \partial r_{2} }
	\right|_{r_{2}=U_{2}(\cdot)}
		\left|	U'_{2}(\cdot) \right|,
\end{eqnarray*}
that is, the density of the random variable \(U_{2}^{-1}(\eta_{3})\), where \(\eta_{3}\) has a distribution function 
\(\tilde{F}(r_{2})\). Since \eqref{fp} holds, we get that 
\begin{eqnarray*}
	\frac{\partial^{2} F (r_{1},r_{2})}{\partial r_{1} \partial r_{2} }
	=
	\frac{ \partial f^{*}(r_{1}, r_{2})} {\partial r_{2}}
	p^{*} ( f^{*}(r_{1}, r_{2})),
\end{eqnarray*}
and therefore \(\eta_{3}\) has the same distribution as \(h^{*}(r_{1}, \eta_{4})\), where \(\eta_{4}\) has distribution with density  \(p^{*}(\cdot)\).

Finally we get the following representation for the L{\'e}vy measure \(\nu\):
\begin{eqnarray*}
\label{nuBplus}
	\nu (B) &=&  
	\int_{\R_{+}} 
	\Biggl[
		\int_{B_{1}}
			p_{1}\Bigl( 
				u \; ; 
				U_{1}^{-1} (r_{1})
			\Bigr)
			du
			\; \cdot \;
		\P \Bigl\{
			\eta_{1} \sqrt{ U_{2}^{-1}(h^{*}(r_{1},\eta_{4}))}
			\in B_{2}
		\Bigr\}	
		\Biggr]	
	 dr_{1}.
\end{eqnarray*}
This representation motivates to define  the function \(\vec{H}\) by 
\begin{eqnarray*}
\vec{H}(r, \vec{D}) =  \left( 
\begin{matrix}
	\sqrt{U_{1}^{-1}(r)} \cdot  D_{1}
	   \\
	   D_{2} \sqrt{ U_{2}^{-1}(h^{*}(r,D_{3}))}
 \end{matrix}
\right),
 \end{eqnarray*}
with  \(\vec{D} = (D_{1}, D_{2}, D_{3})\), where \(D_{1}, D_{2}\) have standard normal distribution, and \(D_{3}\) has a distribution with density function \(p^{*}(\cdot)\).  This observation completes the proof.

\end{proof}
\begin{rem}
\label{cop}
In the context of ordinary copulas, it is common to introduce the homogeneous copula \(C_{H}^{(k)}\) of order \(k\) by
\begin{eqnarray}
\label{CH}
	C_{H}^{(k)} (k u, k v) = k^{\alpha} C_{H}(u,v), \qquad \forall \; k,u,v >0,
\end{eqnarray}
see, e.g., \cite{Nelsen}. Substituting \(u=v=1\), we get \(C_{H}(k,k)=k^{\alpha}\). Therefore, taking into account the Fr{\'e}chet bounds, we arrive at the inequality \[\max(2 k -1,0 ) \leq k^{\alpha} \leq k,\] which yields that \(\alpha \in [1,2]\). Moreover, it turns out that the class of  homogeneous ordinary copulas coincides with   Cuadras-Aug{\'e} family. More precisely, 
\begin{eqnarray*}
		C_{H}^{(k)}(u,v) 
=
		\left(
			\min(u,v)
		\right)^{2-\alpha}
		\left(
			u v
		\right)^{\alpha-1}, \quad u,v \in [0,1],
\end{eqnarray*}
see Theorem~3.4.2 from \cite{Nelsen}.
Returning  to L{\'e}vy copulas, we realize that similar to \eqref{CH} equality
\begin{eqnarray*}
	F_{H} (k u, k v) = k^{\alpha} F_{H}(u,v), \qquad \forall \; k,u,v >0
\end{eqnarray*}
is possible only in case \(\alpha=1\). In fact, taking limit as \(v \to \infty\), we get the equality \(k u = k^{\alpha} u, \; \forall u\), which leads to trivial conclusion \(\alpha=1\). This argument yields the definition of homogeneous L{\'e}vy copula \eqref{LC}.

\end{rem}
 \begin{rem}
 \label{remm}
  Let us shortly show that the function \(g(r_{1} ; \; \cdot)\) defined by \eqref{g} is a density function for any \(r_{1}\). In fact, as it was mentioned before, the function 
 \(\tilde{F}(r_{2}) = \partial F (r_{1}, r_{2}) / \partial r_{1}\) is a distribution function, and moreover \(\partial^{2} F (r_{1}, r_{2}) / \partial r_{1} \partial r_{2} \) is the denisity function of this distribution. Therefore,  \(g( r_{1}; \; \cdot) \geq 0 \), and 
 \begin{eqnarray*}
 \int_{\R} g(r_{1}; \; v ) dv &=& 
  \int_{\R_{+}}
  \Biggl[
  \int_{\R}
			p_{1}\Bigl( 
				U_{2}^{-1} (r_{2}); \; v
			\Bigr)
			dv
\Biggr]
	 	 	\frac{\partial^{2} F (r_{1}, r_{2})}{\partial r_{1} \partial r_{2} }
			dr_{2}\\ 
&=& 
\int_{\R_{+}}
\frac{\partial^{2} F (r_{1}, r_{2})}{\partial r_{1} \partial r_{2} }
			dr_{2} =1.
 \end{eqnarray*}
  \end{rem}
 
 \begin{rem}
 \label{rem54}
It is a worth mentioning that the right way to truncate series in \eqref{res1}-\eqref{res2} is to fix some \(r\) and keep \(N(r) = \inf_{i} \left\{ \Gamma_{i} \geq r \right\}\) terms, see \cite{ContTankov} for details.

 \end{rem} 
%
%
%
%

\section{Empirical analysis}
In this section, we consider the  following model:
\begin{eqnarray}
\label{xs}
\vec{X}(s):=\left( \tW_{1}(\T_{1}(s)), \tW_{2}(\T_{2}(s)) \right)
\end{eqnarray}
where 
\begin{eqnarray}
\label{xs2}
\tW_{i}(t)=\mu_{i} t+\sigma_{i} W_{i}(t), \qquad i=1,2,
\end{eqnarray}
$W_{1}(t), W_{2}(t)$ are two independent Brownian motions, \(\mu_{1}, \mu_{2} \in \R\), \( \sigma_{1}, \sigma_{2} \in \R_{+}\), and \(\left( \T_{1}(s), \T_{2} (s)\right) $ is a two-dimensional subordinator. The dependence between \(\T_{1}(s)\) and \(\T_{2}(s)\) is described via some L{\'e}vy copula \(F(\cdot, \cdot; \delta)\) which belongs to a  class parametrized by \(\delta \in \R.\) The marginal subordinators \(\T_{1}(s)\) and \(\T_{2}(s)\) belong to some parametric class of L{\'e}vy processes. The corresponding parameters are denoted by \(\theta_{1}\) and \(\theta_{2}\).

In what follows, we will apply this model to the  modeling of stock returns. In this context, \(\vec{X}(s)\) represents the returns of two stocks traded on the Nasdaq, and 
\(\left(\T_{1}(s), \T_{1}(s)\right)\) are cumulative numbers of trades of these stocks. Our approach can be considered as a generalization of the paper \cite{AneGeman}, where the one-dimensional time-changed Brownian motion is used for representing one-dimensional stock returns. In \cite{AneGeman}, it is shown that the cumulative number of trades is a good approximation of business time. More precisely, the authors showed that the theoretical moments of the subordinator almost perfectly coincide with the empirical moments of cumulative number of trades.

Our simulation study consists of three steps.
\begin{enumerate}
\item \textit{Estimation of the L{\'e}vy copula between one-dimensional subordinators.} First, we assume some parametric structure of the L{\'e}vy copula between \(\T_{1}(s)\) and \(\T_{2}(s)\) and estimate the parameters of this structure. Our estimation procedure is motivated by the research \cite{Esm}, \cite{Esm2} and is described below in Section~\ref{st1}.

\item \textit{Estimation of the parameters of the processes \(\tW_{1}(t)\) and \(\tW_{2}(t)\).} On this stage, we apply methodology described in \cite{AneGeman} to the processes   \(\tW_{1}(t)\) and \(\tW_{2}(t)\), and get the estimators of the parameters \(\mu_{1}, \mu_{2},  \sigma_{1}, \sigma_{2}\). This part of the empirical analysis is given in Section~\ref{st2}.

 \item \textit{Applying simulation techniques.} Since we already have the estimates of the L{\'e}vy copula between the subordinators and the parameters \(\mu_{1}, \mu_{2},  \sigma_{1}, \sigma_{2}\), we can apply the main theoretical result of the current paper presented in  Theorem~\ref{thm2}. The algorithm is described in Section~\ref{st3}. Some graphs are given in the appendix.
\end{enumerate}

\subsection{Description of the data}
 We examine 10- and 30-minutes Cisco, Intel and Microsoft prices traded on the Nasdaq over the period  from the 25. August 2014 till the 21. November 2014. For each equity we have time variable, price variable and number of trades. The length of time series is 832 observations for 30- minutes data and 2496 observations for 10- minutes data.
Before the analysis, last observations of each trading day were deleted due to the abnormally small number of trades. We suppose that it is related to the microstructure of the market.  

Some descriptive statistics of the data are given  in  Tables~\ref{t1} and \ref{t2} (see Appendix~\ref{AA}).

\subsection{Step 1. L{\'e}vy copula estimation}
\label{st1}
The techniques for the parametric estimation of L{\'e}vy copula are not well-described in the literature and are mainly known for compound poisson processes (see \cite{Avanzi}, \cite{Esm}, \cite{Esm2}).
 In this research, we also assume that subordinators \(\T_{1}(t)\) and \(\T_{2}(t)\)  are compound poisson processes (CPP) with positive jumps, that is, 
\begin{eqnarray}
\label{cpp}
\T_{1}(t) = \sum_{i=i}^{N_1(t)} X_i,  \qquad
\T_{2}(t) = \sum_{j=1}^{N_2 (t)} Y_j ,
\end{eqnarray}
where $X_i$ and $Y_i$ are i.i.d random variables with densities $f_1(x; \theta_{1})$ and $f_2(x; \theta_{2})$ with support in \(\R_{+}\), $N_1(t)$  and $N_2(t)$ are Poisson processes with intensities \(\lambda_{1}\) and \(\lambda_{2}\) resp.   

At first glance,  \eqref{cpp}  seems to be a strong assumption. However,  the L{\'e}vy processes with truncated jumps, that is processes in the form 
\begin{eqnarray*}
	J_{t} =
	\sum_{\substack{ 0 \le s \le t\\ |\Delta Z_{s}| \ge c} }
	\Delta Z_{s}, \qquad \Delta Z_{s} = Z_{s}-Z_{s-},
\end{eqnarray*}	 
where \(Z\) is a (multi-dimensional) L{\'e}vy process, \(c\) is a positive constant,   are compound poisson processes (see \cite{ContTankov}, \cite{Sato}). Moreover, the  truncation of the jumps is a well-used procedure for different simulation and estimation techniques, and is a quite natural tool in the context of stock data (see \cite{ContTankov}, \cite{Esm2}). 

Two-dimensional CPP could be represented in the following way:
\begin{eqnarray*}
\T_1(t)&=&\T^{\bot}_1(t)+\T^{\|}_1 (t),\\
\T_2(t)&=&\T^{\bot}_2(t)+\T^{\|}_2 (t), 
\end{eqnarray*}
where $\T^{\bot}_1$, $\T^{\bot}_2(t)$, $(\T^{\|}_1 (t), \T^{\|}_2 (t))$ are independent compound Poisson processes, and moreover
\begin{itemize}
	\item the L{\'e}vy measure of the process \(\T^{\bot}_1(t)\) is equal to \(\eta(\cdot, 0)\), that is, the component \(\T^{\bot}_1(t)\)  represents the jumps of the process \(\T_1(t)\) which occur independently of the process \(\T_2(t)\);
	\item analogously, the L{\'e}vy measure of the process \(\T^{\bot}_2(t)\) is equal to \(\eta(0, \cdot)\), that is, the component \(\T^{\bot}_2(t)\)  represents the jumps of the process \(\T_2(t)\) which occur independently of the process \(\T_1(t)\);
	\item the L{\'e}vy measure of the joint process \(\left(\T^{\|}_1 (t), \T^{\|}_2 (t)\right)\) is equal to \(\eta(B)\) for all sets \(B\) such that \((0,y)\) and \((x,0)\) doesn't belong to \(B\) for any real \(x,y\), that is,  \(\left(\T^{\|}_1 (t), \T^{\|}_2 (t)\right)\) represents the simultaneous jumps of the processes \(\T_1(t)\) and  \(\T_2(t)\).
\end{itemize}
Shortly speaking, we decompose the 2-dimensional CPP into the  jump independent parts $ \left( \T^{\bot}_1, \T^{\bot}_2 \right)$ and jump dependent parts $(\T^{\|}_1,\T^{\|}_2)$. In other words,  first part represents positive jumps only in one coordinate and second part represents positive jumps in both coordinates.  We denote
by \(\Pi^{\bot}_1, \Pi^{\bot}_2, \Pi^{\|} \)  the L{\'e}vy measures of the processes $\T^{\bot}_1, \T^{\bot}_2, (\T^{\|}_1,\T^{\|}_2)$ resp., by \(\lambda^{\bot}_1, \lambda^{\bot}_2, \lambda^{\|} \) - the intensities of the corresponding processes, and by $n^{\bot}_1, n^{\bot}_2, n^{\|}$ - the total number of jumps occuring only in the observed processes up to some fixed time \(T\). 

 The characterisitc function of two-dimensional CPP can be decomposed as follows (see \cite{ContTankov}):
\begin{multline*}
E[\exp(iz_1 \T_1(t)+iz_2 \T_2(t))] =\\ 
 \exp \left\{ t \int_{R}^{} (\exp (iz_1 x)-1) \Pi^{\bot}_1 (dx) +
t \int_{R}^{} (\exp (iz_2 y)-1) \Pi^{\bot}_2(dy) +\right.\\
\left.
+t \int_{R^{2}}^{} (\exp(iz_1x+iz_2y)-1) \Pi^{\|}(dx \times dy) \right\} =\\
E\left[ \exp(iz_1 \T_1^{\bot}(t)) \right] E\left[\exp(iz_2 \T_2^{\bot}(t))  \right] E\left[\exp(iz_1 \T_1^{\|}(t)+iz_2 \T_2^{\|}(t))\right].
\end{multline*}

In \cite{Esm},  Esmaeili and Kluppelberg applied the MLE approach to estimate the parameters in this model.  According to the Theorem~4.1 from \cite{Esm}, the likelihood of the bivariate compound poisson process is given by: 
\begin{eqnarray}
\label{likelihood}
 L(\lambda_1,\lambda_2,\theta_1,\theta_2,\delta)= I_{1} \cdot I_{2} \cdot I_{3}
\end{eqnarray}
 where
 \begin{eqnarray*}
I_{1} &=& (\lambda_1)^{n^{\bot}_1} \exp(-\lambda^{\bot}_1 T) \\&& \hspace{2cm} \cdot\prod_{i=1}^{n^{\bot}_1} \left[ f_1(\overset{\thicksim}{x_i},\theta_1) \left( 1- \frac{\partial }{\partial u} F(u,\lambda_2;\delta) |_{u=\lambda_1 
\bar{F}_1 (\overset{\thicksim}{x_i};\theta_1)
} \right) \right],\\
I_{2} &=&  (\lambda_2)^{n^{\bot}_2} \exp(-\lambda^{\bot}_2 T) \\
&& \hspace{2cm} \cdot\prod_{i=1}^{n^{\bot}_2} \left[ f_2(\overset{\thicksim}{y_i},\theta_2) \left( 1- \frac{\partial }{\partial v} F(\lambda_1,v;\delta) |_{v=\lambda_2 
\bar{F}_2 (\overset{\thicksim}{y_i};\theta_2)
} \right) \right],\\
I_{3} &=& (\lambda_1 \lambda_2)^{n^{\|}} \exp(-\lambda^{n^{\|}}T) \\&& \hspace{2cm}\cdot \prod_{i=1}^{n^{\|}} 
\left[ f_1 (x_i;\theta_1) f_2(y_i;\theta_2) 
\cdot
\frac{\partial^{2}}{\partial u \partial v} F(u,v,\delta)|_{\substack{u=\lambda_1 
\bar{F}_1 (x_i;\theta_1)
,\\v=\lambda_2 \bar{F}_2 (y_i;\theta_2)}} \right].
\end{eqnarray*}
where $\overset{\thicksim}{x_i}$ and $\overset{\thicksim}{y_i}$ are jumps in the first and the second component occuring at different time, $x$ and $y$ are jumps for both components occuring at the same time, \( \bar{F}_{i}(\cdot; \theta_{i}) := \int_{\cdot}^{+\infty} f_i (u ;\theta_i) du, \; i=1,2.\)

Formula \eqref{likelihood} is based on a couple of simple facts, which directly follows from our construction. First, note that distance between jump moments for both components are independent and identically distributed exponential random variables with parameters $\lambda_1$ and $\lambda_2$. Second, one can show that
\begin{eqnarray}
\label{tails}
U_i=:U^{\bot}_i+U^{\|}_i,
\end{eqnarray}
where $U_i, \; i=1,2$ - marginal tail integral, $U^{\bot}_i$ and $U^{\|}_i$ are one-dimensional tail integrals of independent and dependent parts. Third, it is a worth mentioning that
\begin{eqnarray}
\label{lambda}
\lambda^{\|}=C(\lambda_1,\lambda_2;\delta).
\end{eqnarray}
From (\ref{tails}) and properties of two-dimensional tail integral and L{\'e}vy copula we get that for any positive \(x,y\)
\begin{eqnarray*}
\lambda^{\bot}_1 (x) (1-F^{\|}_1(x)) &=& \lambda_1 \bar{F}_{1}-\textit{C}(\lambda_1 \bar{F}_{1}(x),\lambda_2;\delta))\\
\lambda^{\bot}_2 (x) (1-F^{\|}_2(x)) &=& \lambda_2\bar{F}_{2}(x) -\textit{C}(\lambda_1,\lambda_2 \bar{F}_{2}(x);\delta))\\
\lambda^{\|} F(x,y) &=&\textit{C}(\lambda_1 \bar{F}_{1}(x),\lambda_2 \bar{F}_{2}(y);\delta).
\end{eqnarray*}
For numerical example, we model the process of two-dimensional cumulative number of trades as a compond poisson process with exponential jumps. Denote by $\theta_1$ and $\theta_2$ the parameters of  the jump sizes densities, that is, \[f_i(x;\theta_i)=\theta_i \exp(-\theta_i x) \quad\mbox{for} \quad x>0, \quad i=1,2\] is the jump density for the \(i\)-th component. The dependence between \(\T_{1}\) and  \(\T_{1}\) is described by a Clayton  L{\'e}vy copula with  parameter $\delta$ . Clayton copula is homogeneous copula and perfectly fits the conditions of the Theorem 5.1 (see Example 1 on page~\pageref{clevy}).

The likelihood function of the continuously observed two-dimensional CPP process $(\T_1(t), \T_2(t))$ can be written in the following form, assuming that jumps occur at each moment for both components (there are no time intervals without trades in our data due to the fact that Cisco, Intel and Microsoft are liquid securities):
\begin{multline}
\label{L}
L(\lambda_1,\lambda_2,\theta_1,\theta_2,\delta) \\= \left( (1+\delta) \theta_1 \theta_2 (\lambda_1 \lambda_2)^{\delta+1} \right)^{n} 
 \exp\left\{-\lambda^{\|}T -(1+\delta) \left( \theta_1 \sum_{i=1}^{n} x_i +\theta_2 \sum_{i=1}^{n} y_i\right) 
 \right\} \\
\cdot\prod_{i=1}^{n} \left( \lambda^{\delta}_1 \exp(-\theta_1 \delta x_i)+\lambda_2^{\delta} \exp(-\theta_2 \delta y_i) \right)^{-\frac{1}{\delta}-2}.
\end{multline}
The results of the  numerical optimization of \(\ln L(\lambda_1,\lambda_2,\theta_1,\theta_2,\delta)\) are presented in the  Table~\ref{t3}.

\vspace{1cm}

\begin{table}
\caption{}\label{t3}
\begin{tabular}{|r|r|r|r|r|r|r|}
\hline
                                             \multicolumn{ 7}{|c|}{Estimated parameters} \\
\hline
      Pair &     $\theta_1$ &     $\theta_2$ &      $\delta$ &    $\lambda_1$ &    $\lambda_2$ & Log likelihood value \\
\hline
                                               \multicolumn{ 7}{|c|}{30-minutes returns} \\
\hline
Csco vs Int &       0,29 &       0,14 &       2,21 &      24,91 &      14,69 &    5161,43 \\
\hline
Csco vs Msf &       0,23 &       0,14 &       2,71 &      16,39 &      17,21 &    5196,93 \\
\hline
Int vs Msf &       0,14 &       0,17 &       2,38 &      14,41 &      24,68 &    5579,32 \\
\hline
                                               \multicolumn{ 7}{|c|}{10-minutes returns} \\
\hline
Csco vs Int &       0,85 &       0,43 &       1,76 &      74,18 &      48,60 &   10299,11 \\
\hline
Csco vs Msf &       0,71 &       0,42 &       2,11 &      52,66 &      55,78 &   10406,96 \\
\hline
Int vs Msf &       0,42 &       0,49 &       2,00 &      46,22 &      72,48 &   11511,90 \\
\hline
\end{tabular}  
\end{table}

\subsection{Step 2. Estimation in time-changed model}
\label{st2}
 Let us shortly recall the model of stochastic time change  in one-dimensional case. Denote by $P_t$ a equity price at moment t, and the returns by 
\begin{eqnarray}
Y_t=\log\left(\frac{P_t}{P_{t-1}}\right).
\end{eqnarray}
The  main  idea of the pioneer research \cite{AneGeman} is to show that  
\begin{eqnarray}
\label{xtau}
Y_{t}=\tW(\tau(t)),
\end{eqnarray}
where $\tW(t)=\mu t+\sigma W(t)$ with Brownian motion \(W_{t}\), and $\tau(t)$ is the cumulative number of trades up to time \(t\). It is a well- known fact that on one side, financial returns generally are not normally distributed (e.g. they are fat tailed), and on the other side,  normality hypothesis is very convenient tool in finance, e.g. in mean-variance paradigm. In this context, formula \eqref{xtau} shows that returns are in some sense normal in business time, which differs from calendar time.

Returning to our model \eqref{xs}-\eqref{xs2}, we now consider the problem of statistical estimation of the 
 parameters \(\mu_{1}, \mu_{2},  \sigma_{1}, \sigma_{2}\). This task can separately solved for both components of the vector \(\vX(s)\) by the method of moments. Assuming that \(\tau\) is a CPP with intensity \(\lambda \) and  jumps distributed by exponential law with parameter \(\theta\), we get 
\begin{eqnarray*}
E[Y_t]=\mu \frac{\lambda t}{\theta}, \qquad 
\Var[Y_t]=\frac{\sigma^{2} \lambda t}{\theta}+\frac{2 \mu^{2}  \lambda t}{\theta^{2}}.
\end{eqnarray*}
Solving the system of equations
\begin{eqnarray*}
E[Y_t] = \widehat{E[Y_t]}, \qquad  \Var[Y_t] = \widehat{\Var[Y_t]},
\end{eqnarray*}
we arrive at the following estimates of the parameters $\mu$ and $\sigma^{2}$:
\begin{eqnarray*}
\hat{\mu}=\frac{\theta \widehat{E[Y_t]}}{\lambda t}, \qquad
\hat{\sigma^{2}}=\frac{\widehat{\Var[Y_t]}-2 \hat{\mu}^{2} \lambda t/ \theta}{\lambda t }.
\end{eqnarray*}
Estimation results are presented in the  Table~\ref{t3}.

\begin{table}
\caption{}\label{t4}
\begin{tabular}{|r|r|r|r|r|}
\hline
      Pair &    $\mu_1$ &     $\mu_2$ & $\sigma^{2}_1$ &    $\sigma^{2}_2$ \\
\hline
                        \multicolumn{ 5}{|c|}{30-minutes data} \\
\hline
 csco intc &   1,94E-08 &   7,25E-09 &   3,45E-09 &   4,93E-09 \\
\hline
 csco msft &   2,31E-08 &   1,16E-08 &   4,10E-09 &   3,05E-09 \\
\hline
 intc msft &   7,33E-09 &   1,00E-08 &   4,99E-09 &   3,84E-10 \\
\hline
                        \multicolumn{ 5}{|c|}{10-minutes data} \\
\hline
 csco intc &  -6,00E-09 &  -7,64E-10 &   5,17E-10 &   9,56E-10 \\
\hline
 csco msft &  -7,05E-09 &  -2,63E-09 &   6,07E-10 &   4,14E-10 \\
\hline
 intc msft &  -7,85E-10 &  -2,35E-09 &   9,82E-10 &   3,69E-10 \\
\hline
\end{tabular}  
 \end{table}

\subsection{Step 3. Simulation techniques}
\label{st3}
Here we show the performance of our approach introduced in Theorem~\ref{thm2}. Our goal is  to simulate two-dimensional time-changed L{\'e}vy process:
\begin{eqnarray*}
\vec{X}(s) &=&\left( \tW_{1}(\T_{1}(s)), \tW_{2}(\T_{2}(s)) \right)\\
&=&
\Bigl(
	\mu_{1} \T_{1}(s) +\sigma_{1} W_{1}(\T_{1}(s)),
	\;\;
	 \mu_{2} \T_{2}(s) +\sigma_{2} W_{2}(\T_{2}(s))
\Bigr).
\end{eqnarray*}
Our simulation algorithm consists of the following steps:
\begin{enumerate}
\item Model an independent sequence of jump times of a standard Poisson process $\Gamma_i$: 
$$
\Gamma_i=\sum_{j=1}^{i} T_j<r,
$$
where $r$ determines the truncation level, $T_j$ is a standard exponential random variable.

\item Model $k$ independent  standard normal random variables $G_i^{(1)}$ and $G_i^{(2)}$, where $i=1,\dots ,k $.
\item Model $k$ independent uniform random variables $R_i$ on $[0,1]$, where $i=1,\dots ,k$.
\item Model $k$ independent random variables $G^{(3)}_i$ with distribution function $F(z),$ which is equal to 
\begin{eqnarray}
F(z)=\frac{1}{(z^{-\theta}+1)^{\frac{(1+\theta)}{\theta}}}
\end{eqnarray}
 by the method of inverse function, that is $G^{(3)}_i \equiv F^{-1}(\xi_i)$, where $\xi_i$ are  independent uniform random variables on $[0,1]$.

\item Model subordinated Brownian motions by (truncated) series representation:
	\begin{eqnarray*}
		Z_{1} (s) &:=&  \sum_{i=1}^{k} \sqrt{U_{1}^{-1}(\Gamma_{i})} \cdot  G^{(1)}_{i} \cdot I\left\{ R_{i} \leq s \right\},\\ 
		Z_{2} (s)  &:=&  \sum_{i=1}^{k}   G^{(2)}_{i}  \sqrt{ U_{2}^{-1}\left(
			h^{*}(\Gamma_{i},G^{(3)}_{i} )
		\right)}  \cdot  I\left\{ R_{i} \leq s \right\},
	\end{eqnarray*}		
where the generalized inverse functions of \(U_{i}, \; i=1,2\) have the form
\begin{eqnarray}
{U_i}^{(-1)}(x)=
\left\{
\begin{aligned}
-\frac{1}{\theta_i} \log (\frac{x}{\lambda_i}), & \quad \text{for}\;\;  x \leq \lambda_i,\\
0, & \quad \text{for} \;\;x >\lambda_i,
\end{aligned}
\right.
\end{eqnarray}
and $h^{*}(x,y)$ is equal to 
\(
h^{*}(x,y)=x y 
\).
\item Model two-dimensional subordinator $(\T_1(s),\T_2(s))$ with Clayton-L{\'e}vy copula and compound poisson margins (with exponential jumps) by series representation (see \cite{ContTankov}, algorithm 6.13).
\item Resulting trajectory is a linear transform of subordinator and subordinated brownian motion:
\begin{eqnarray}
\label{resulting}
X_1(s)=\hat{\mu_1} \T_{1}(s)+\hat{\sigma^{2}_1} Z_1(s)\\
\label{resulting2}
X_2(s)=\hat{\mu_2} \T_{2}(s) +\hat{\sigma^{2}_2} Z_2(s)
\end{eqnarray}
\end{enumerate}

Typical trajectories of simulated processes are presented in the Appendix. Figures~\ref{30MIN51} and \ref{10MIN51} display trajectories for time-changed brownian motions modeled by Theorem~\ref{thm2} for 30 and 10 minutes data. Figures \ref{30MIN51sub} and \ref{10MIN51sub} show typical trajectories for suborninators modelled as compound poisson process with exponential jumps for 30 and 10 minutes data. Finally, Figures \ref{30MINY} and \ref{10MINY} display resulting trajectories for the two-dimensional process $\vX$ calculated by \eqref{resulting}-\eqref{resulting2}.

\subsection{Further research}
One interesting question, which was not addressed before, is to compare the L{\'e}vy copula between simulated process $\left( \tW_{1}(\T_{1}(s)), \tW_{2}(\T_{2}(s)) \right)$ and copula between subordinators $\left( \T_{1}(s), \T_{2}(s) \right)$. This question is motivated by the paper \cite{Eberlein}, where some relations between the corresponding correlation coefficients are given.

In this paper, we would like to visually compare the copulas. The nonparametric estimation of the L{\'e}vy copula between L{\'e}vy processes \(X^{(1)}\) and \(X^{(2)}\) has been recently studied in \cite{BV}. The proposed estimator for any $x_1\ge 0$ and $x_2 \ge 0$ is equal to 
\begin{eqnarray}
\label{Copula_est}
\hat{F}(x_1,x_2)=\sum_{k=1}^{n} I\{\widehat{U}_{1,n}(\Delta^n_kX^{(1)})\leq x_1 ,\U_{2,n}(\Delta^n_kX^{(2)}) \leq x_2\},
\end{eqnarray}
where by      $$\Delta^n_kX^{(i)}=X_{k\Delta_n}^{(i)}-X_{(k-1)\Delta_n}^{(i)}, \quad k=1..n, \; i=1,2,$$
we denote the increments of the processes \(X^{(1)}\) and \(X^{(2)}\), and
\begin{eqnarray}
\label{OnedTail_est}
\U_{n,i}(x)=\frac{1}{n\Delta_n}\sum_{k=1}^n I \{ \Delta^n_kX^{(1)}\geq x \}, \quad i=1,2,
\end{eqnarray}
are the non-parametric estimators of the tail integrals of the underlined L{\'e}vy processes.

We applied this methodology to the simulated process $\tW_{1}(\T_{1}(s))$ and \(\tW_{2}(\T_{2}(s))\)  and got the L{\'e}vy-copula estimate (see Figure~\ref{plotts}). 

\bibliographystyle{imsart-nameyear}
\bibliography{Panov_bibliography}
\appendix
\section{Descriptive statistics}
\label{AA}
\begin{landscape}
\begin{table}
\caption{} \label{t1}
\begin{center}
\begin{tabular}{|r|r|r|r|r|r|r|r|r|}
\hline
                 \multicolumn{ 9}{|c}{Descriptive statistics for number of trades, in thousands}  \\
\hline
Variable &     mean &    stdev &      min &      max &   median &    $m_2$ &   $m_3$ &   $m_4$    \\
\hline
  30csco &     5,13 &     3,12 &     1,19 &    33,41 &     4,24 &     9,74 &    74,52 &  1357,62  \\
\hline
   30int &     7,83 &     6,22 &     2,09 &    79,24 &     6,02 &    38,64 &  1093,33 & 57900,47  \\
\hline
   30msf &     8,87 &     5,53 &     1,46 &    59,85 &     7,53 &    30,54 &   539,90 & 20579,57  \\
\hline
                                                                          \multicolumn{ 9}{|c}{}  \\
\hline
  10csco &     1,71 &     1,20 &     0,26 &    13,01 &     1,41 &     1,44 &     5,11 &    36,57  \\
\hline
   10int &     2,61 &     2,32 &     0,52 &    39,43 &     2,00 &     5,37 &    69,65 &  1759,46  \\
\hline
   10msf &     2,96 &     2,11 &     0,35 &    39,80 &     2,43 &     4,46 &    42,90 &  1022,25 \\
\hline
\end{tabular}
\end{center}  
\end{table}
\vspace{2cm}
\begin{table}
\caption{}\label{t2}
\begin{center}
\begin{tabular}{|r|r|r|r|r|r|r|r|r|}
\hline
                                                         \multicolumn{ 9}{|c|}{Descriptive statistics for returns} \\
\hline
   Returns &       mean &      stdev &        min &        max &     median &         $m_2$ &         $m_3$ &         $m_4$ \\
\hline
    30csco &   1,06E-04 &   4,35E-03 &  -4,81E-02 &   3,01E-02 &  -1,12E-04 &   1,89E-05 &  -6,04E-08 &   1,25E-08 \\
\hline
     30int &   5,00E-05 &   5,84E-03 &  -6,81E-02 &   4,16E-02 &   0,00E+00 &   3,40E-05 &  -2,59E-07 &   4,42E-08 \\
\hline
     30msf &   9,34E-05 &   4,97E-03 &  -5,65E-02 &   3,86E-02 &   0,00E+00 &   2,46E-05 &  -1,39E-07 &   2,67E-08 \\
\hline
                                                                                           \multicolumn{ 9}{|c|}{} \\
\hline
    10csco &  -3,36E-05 &   1,70E-03 &  -1,17E-02 &   2,26E-02 &   0,00E+00 &   2,90E-06 &   7,01E-09 &   2,70E-10 \\
\hline
     10int &  -5,50E-06 &   2,62E-03 &  -3,17E-02 &   5,56E-02 &   0,00E+00 &   6,87E-06 &   5,94E-08 &   4,86E-09 \\
\hline
     10msf &  -2,22E-05 &   1,87E-03 &  -2,00E-02 &   3,09E-02 &  -5,50E-06 &   3,49E-06 &   8,84E-09 &   5,28E-10 \\
\hline
\end{tabular}
\end{center}
\end{table}  
\end{landscape}

\section{Graphs}

\begin{figure}[bh]
\noindent
\centering
\includegraphics[width=140mm]{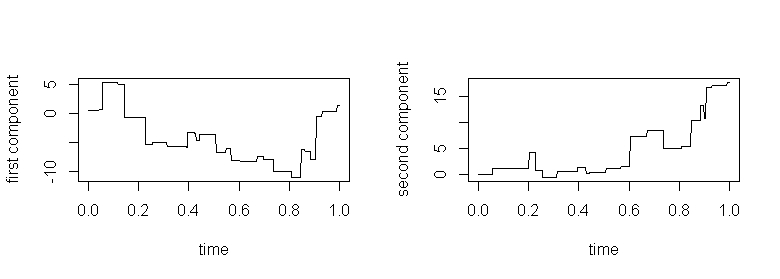}
\caption{Time-changed brownian motion. Subordinators are CPP with exponential jumps. Parameters are estimated from the Cisco and Intel 30-minutes data}
\label{30MIN51}
\end{figure}

\begin{figure}[bh]
\noindent
\centering
\includegraphics[width=140mm]{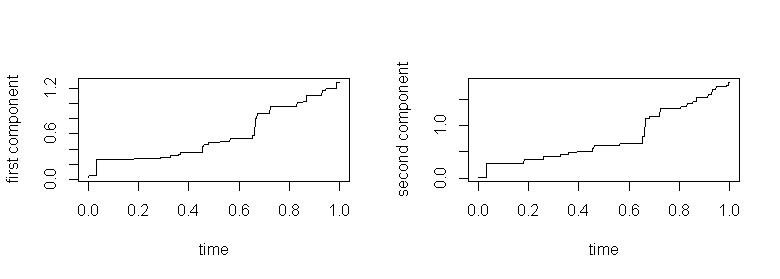}
\caption{Subordinators for 30 minute data. Parameters are estimated from the Cisco and intel 30-minute data}
\label{30MIN51sub}
\end{figure}

\begin{figure}[bh]
\noindent
\centering
\includegraphics[width=140mm]{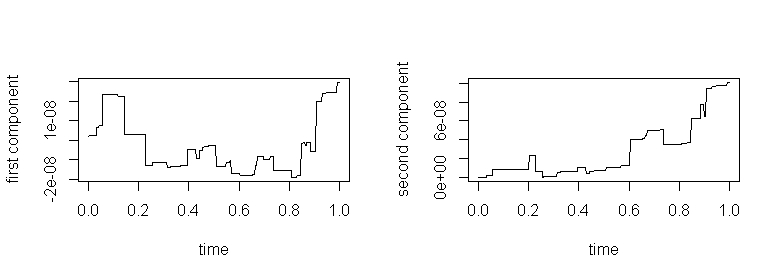}
\caption{Resulting trajectory of process $Y(t)$ for 30 minute data}
\label{30MINY}
\end{figure}

\begin{figure}[bh]
\noindent
\centering
\includegraphics[width=140mm]{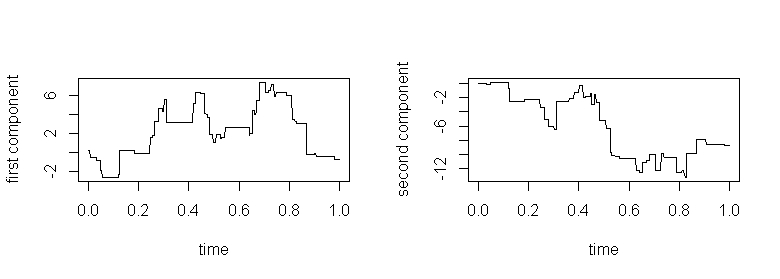}
\caption{Time-changed brownian motion. Subordinators are CPP with exponential jumps. Parameters are estimated from the Cisco and Microsoft 10-minutes data}
\label{10MIN51}
\end{figure}

\begin{figure}[bh]
\noindent
\centering
\includegraphics[width=140mm]{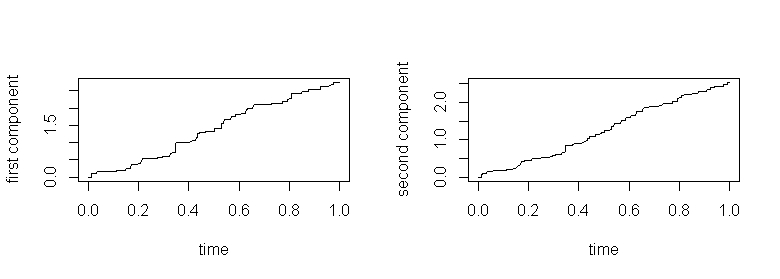}
\caption{Subordinators for 10-minute data. Parameters are estimated from Cisco and Microsoft 10-minute data}
\label{10MIN51sub}
\end{figure}

\begin{figure}[bh]
\noindent
\centering
\includegraphics[width=140mm]{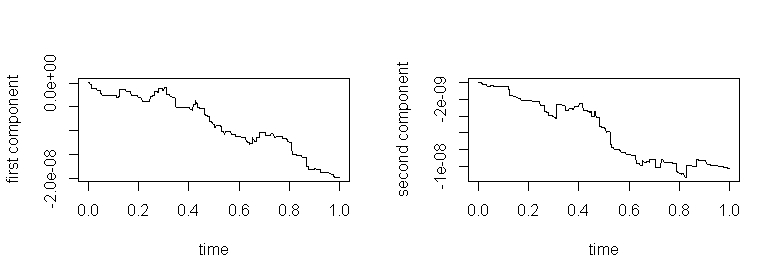}
\caption{Resulting trajectory of process $Y(s)$ for 10- minute data}
\label{10MINY}
\end{figure}

\begin{figure}[bh]
\noindent
\centering
\includegraphics[width=140mm]{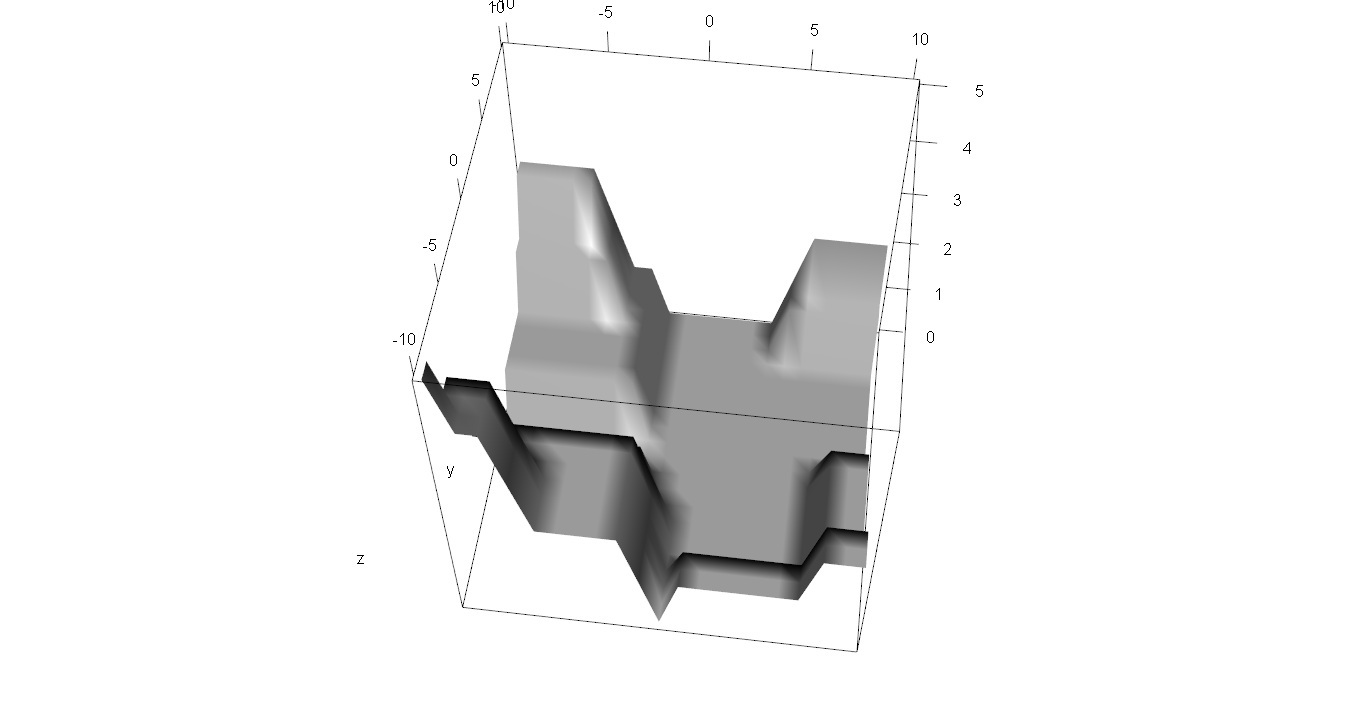}
\caption{Nonparametric estimation of the copula for simulated process $Y(s)$}
\label{plotts}
\end{figure}

\end{document}